\documentclass[twoside,10pt,fleqn]{article}
\usepackage{amsmath}
\usepackage{indentfirst,amsmath,amsfonts,amssymb,amsthm,cite}
\usepackage{fancyhdr}
\usepackage[dvips]{graphics}
\usepackage{graphicx}
\usepackage[dvips]{color}
\newtheorem{Rem}{Remark}[section]

\newtheorem{theo}{Theorem}[section]

\newtheorem{lem}{Lemma}[section]
\newtheorem{iteration lemma}{iteration Lemma}[section]

\newcommand{\ra}{\rightarrow}
\newcommand{\fr}{\frac}

\def\S{\ensuremath {\mathcal S}}

\def\R{\ensuremath {\mathbb R}}

\begin{document}

\setcounter{page}{1}
\newcounter{jie}
\vspace*{0.4cm}
\begin{center}
{\LARGE\bf {Stochastic Nicholson's blowflies delay differential equation with regime switching}$^{*}$}\\[0.6cm]
\normalsize\bf
Yanling Zhu$^{\dagger}$, Kai Wang$^{\dagger,*}$, Yong Ren$^{\ddagger,*}$, Yingdong Zhuang$^{\dagger}$\\[3mm]
 {\footnotesize\it$^{\dagger}$  School of Statistics and Applied
Mathematics, Anhui University of Finance and Economics,\\ Bengbu
233030, Anhui, China\\
 $^{\ddagger}$ Department of Mathematics, Anhui Normal University,
Wuhu 241000, Anhui, China}
\begin{figure}[b]
\footnotesize
\rule[-2.truemm]{5cm}{0.1truemm}\\[1mm]
{$^*$The research of Yanling Zhu was supported by the NSFC (no. 71803001). The research of Kai Wang was supported by the NSF of Anhui
Province (no. 1708085MA17), and the University Natural Science Research Project of
Anhui Province (no. KJ2018A0437). The research of Yong Ren was supported by the NSFC (no. 11871076)
\\Corresponding author.\\
{e-mails:}\;zhuyanling99@126.com(Y. Zhu); wangkai.math\textit{\char64}163.com(K. Wang); brightry@hotmail.com(Y. Ren)
}
\end{figure}
\end{center}
\noindent\hrulefill \newline
\begin{center}
\begin{minipage}{15cm}
\vspace{-3mm} \noindent{\bf Abstract}\ \ In this paper, we investigate the global existence of almost surely positive
solution to a stochastic Nicholson's blowflies delay differential equation with regime switching, and give the estimation of the path. The results presented in this paper extend some corresponding  results in {\it Wang et al. Stochastic Nicholson's blowflies delayed differential equations, Appl. Math. Lett. 87 (2019) 20-26}.


{\bf AMS(2000)}\ \ 34K50;\ \  60H10;\ \  65C30
\end{minipage}
\end{center}
\vspace{1mm} \noindent\hrulefill \newline
\section{Introduction }
The following delay differential equation, known as the famous Nicholson's blowflies model
\begin{equation*}
X'(t)=-\delta X(t)+pX(t-\tau)e^{-aX(t-\tau)}
\end{equation*}
was presented by Gurney et al. (1990) \cite{Gurney} to model the population of Lucilia cuprina({\it the Australian sheep-blowfly}). The biological meaning of the parameters in this model is that: $\delta$
denotes the adult death rate of per capita daily, $p$
denotes the maximum egg production rate of per capita daily, $\tau$ denotes the generation time, and $1/a$
is the size that  the Lucilia cuprina  reproduces at its maximum rate. Since then, the dynamics of this model was investigated by many researchers,  see  Berezansky et al. (2010) \cite{Berezansky} for an overview.
\par
Considering that  the per capita daily adult
death rate $\delta$ is affected by the white noise of the environment, one often  uses $\delta-\sigma dB(t)$ in lieu of $\delta$, where $\sigma$ denotes the intensity of the white noise and $B(t)$ is a
one-dimensional Brownian motion defined on $(\Omega, \{\mathfrak{F}_t\}_{t\geqslant0}, \mathbb{P})$, which is a complete probability space. Wang et al. (2019)\cite{WangW} studied the following stochastic
Nicholson's blowflies delay differential equation:
\begin{equation}\label{wk0}
dX(t) = [-\delta X(t) + pX(t-\tau )e^{-a X(t-\tau)}]dt + \sigma X(t)dB(t),
\end{equation}
with initial conditions
$X(s) = \phi(s), \; \mbox{for}\; s\in[-\tau, 0], \phi \in C([-\tau, 0], [0,+\infty)), \phi(0) > 0.$ Under the condition $\delta>\sigma^2/2$, they prove that equation \eqref{wk0} has a unique global solution on $[-\tau,\infty)$, which is positive a.s. on $[0,\infty)$, and give the estimations  of $\limsup_{t\rightarrow
\infty}\mathbb{E}[X(t)]$ and $\limsup_{t\rightarrow
\infty}\frac 1t \int_0^t \mathbb{E} [X(s)]ds$, respectively.
\par But we find from numerical simulation that the condition $\delta>\sigma^2/2$ is too strict to the nonexplosion of the solution to equation \eqref{wk0}. For example, let $\delta=1$, $p=5$, $\tau= 1$, $a=1$, $\sigma= 2$, then the numerical simulation in Figure 1(left) shows that the solution $X(t)$ is globally existent on $[0, +\infty)$, and oscillating about $N^*=\log p/\delta$, which is the positive equilibrium of the corresponding deterministic equation of  \eqref{wk0}, but $\delta<\sigma^2/2$.
\par Meanwhile, as one knows that besides the white noises there is a so called telegraph noise in the real ecosystem.  A switching process between two or more regimes can be used to illustrate this type of noise. One can model the regime
switching  by a continuous time Markov chain $(r_t)_{t\geqslant0}$, which takes values in a finite state space $\S= \{1, 2, . . . ,m\}$ and satisfies
\vspace{-2mm}
\[\mathbb{P}(r_{t+\vartriangle} = j|r_t = i) = \left\{
                                        \begin{array}{ll}
                                         & q_{ij}\vartriangle + o(\vartriangle),\; \mbox{for}\; i \not= j,\\
                                         & 1 + q_{ij}\vartriangle + o(\vartriangle),\; \mbox{for}\; i = j,
                                        \end{array}
                                      \right.
as \vartriangle \rightarrow 0^+,\]
 where $Q = (q_{ij}) \in R^{m\times m}$ is the infinitesimal generator, and  $q_{ij}\geqslant0$ denotes the transition rate from $i$ to $j$ for $i \not= j$, and $\sum_{j\in\S} q_{ij}=0$,  $\forall\,i \in \S$. In consideration of
  that the white noise and the telegraph noise are deferent types of noise, we assume that the Markov chain $(r_t)_{t\geqslant0}$ is  independent of the Brownian motion $B(t)$. For more details on the theorem of regime switching system, one can refer to Mao et al. (2006) \cite{mao-06}, Khasminskii et al. (2007) \cite{kh-07} and Yin et al. (2009) \cite{yz-09}.
\par Considering both the effects of  white noise and telegraph noise of environment,
we present the following  stochastic Nicholson's blowflies delay differential equation with regime switching
\begin{equation}\label{wk00}
\aligned
dX(t)=&\left[-\delta_{r_t} X(t)+p_{r_t}X(t-\tau_{r_t})e^{-a_{r_t}X(t-\tau_{r_t})}\right]dt
+\sigma_{r_t} X(t)d B(t)
\endaligned
\end{equation}
with initial conditions:
\begin{equation}\label{wk01}X(s) = \phi(s), \; \mbox{for}\; s\in[-\tau, 0], \phi \in C([-\tau, 0], \R^+), 
\end{equation}
where $\tau=\max_{i\in S}\{\tau_{i}\}$, $\R^+=(0,+\infty)$. The contribution of this paper is as follows:
\vspace{1mm}
\par $\bullet$ Model \eqref{wk00} considered in this paper includes regime switching, which is more general than model \eqref{wk0};
\par $\bullet$ We show  the global existence of almost surely positive solution to equation \eqref{wk00} without the restricted condition $\delta_i>\sigma_i^2/2$, $i\in\S$, which extends the one in \cite{WangW};
\par $\bullet$ We give the  estimations of $\limsup_{t\rightarrow
\infty}\mathbb{E}[X^\theta(t)]$, $\limsup_{t\rightarrow
\infty}\fr1t\int_0^t\mathbb{E} [X^\theta(s)]ds$, $\liminf_{t\rightarrow
\infty}X(t)$ and the sample Lyapunov exponent of $X(t)$, which are more general than the ones in literature.
\vspace{1mm}
\par
For simplicity, in the following and next sections we use the following notations:
  \[\aligned
  &\hat {g}=\min_{i\in\S}\{g_i\},\;\,\check{g}=\max_{i\in\S}\{g_i\},\;\, \mathbb{R}^+_0=[0,\infty),\,\;\beta_{r_t}= \,\theta\left
  [\delta_{r_t}+(1-\theta)\sigma_{r_t}^2/2\right],\;\gamma_{r_t}= {p_{r_t}(a_{r_t}e)^{-1}},\,\;\rho_{r_t}=\delta_{r_t}/\sigma_{r_t}^2,\\
&M_{r_t}=\max_{x\in \R^+_0}\{\theta\gamma_{r_t}x^{\theta-1}-\beta_{r_t}x^\theta\},\,\;W_{r_t}=\max_{ x\in \R^+_0}\{(\alpha-\beta_{r_t})x^\theta+\theta\gamma_{r_t}x^{\theta-1}\},\;\,
0<\alpha<\hat\beta,\;\;d_{r_t}=\delta_{r_t}+\sigma^2_{r_t}/2,\\
& \mu_{r_t}=2\delta_{r_t}-p_{r_t}-\sigma^2_{r_t},\,\;\mbox{and}\;\;C_{i}=\left\{
                 \begin{array}{ll}
                   \frac{\sigma^2_{i}-2\delta_{i}+\sqrt{(\sigma^2_{i}-2\delta_{i})^2+4\gamma^2_{i}}}{2}, & \mbox{if}\;\;\delta_{i}\leqslant \sigma^2_{i}/2, \\
                   \frac{\gamma^2_{i}}{2\delta_{i}-\sigma^2_{i}}, & \mbox{if}\;\;\delta_{i}> \sigma^2_{i}/2,
                 \end{array}
               \right.\;\; i\in\S.
\endaligned
  \]
\\  \textbf{Assumption (A)}: The Markov
chain $(r_t)_{t\geqslant0}$ is irreducible with an invariant distribution $\pi=(\pi_i, i\in\S)$.
\vspace{2mm}
\par Before the main results, we first give two lemmas and omit the proof here for saving layouts.
\begin{lem}\label{lem1} Let
\(
f(x)=xe^{-ax},\,a\in \mathbb{R}^+,x\in\R^+_0
\), then $ f(x)\leqslant(ae)^{-1}$. 
\end{lem}
\begin{lem}\label{lem2} For $ a\in \mathbb{R}$, $ b\in\mathbb{R}^+$,  we have 
\(
\frac{ax^2+bx}{1+x^2}\leqslant C(a)
\) for $x\in \mathbb{R},$
where
  $C(a)=\left\{
                  \begin{array}{ll}
                     ({a+\sqrt{a^2+b^2}})/{2}, & a\geqslant0; \\
                    -{b^2}/{4a}, & a<0.
                  \end{array}
                \right.$
\end{lem}
\section{Main Results}
In this section, we will show the global existence of almost surely positive solution of equation \eqref{wk00} with initial value \eqref{wk01}, and give some estimations of the solution $X(t)$.

\begin{theo} \label{Th1} 
 For any given initial value \eqref{wk01}, equation \eqref{wk00}  has a unique solution  $X(t)$ on $[-\tau,\infty)$, which is positive on $\mathbb{R}^+_0$ a.s.
\end{theo}

\begin{proof}
Since all the coefficients of equation \eqref{wk00}  are locally Lipschitz continuous on $\R^+_0$,
then for any given initial value \eqref{wk01}, 
there is a unique maximal local solution $X(t)$ on $[-\tau,\Lambda_e)$,
where $\Lambda_e$ is the explosion time. Assume that the subsystems of \eqref{wk00} switch along the Markov chain $r_0,r_1,r_2,...,r_n$, and the times on each state are $[0,t_0],[t_0,t_1],[t_1,t_2],...,$ $[t_{n-1},\Lambda_e)$, respectively.
 \par {\bf Step 1.} We first show that $X(t)$ is positive almost surely on $[0,\Lambda_e)$.
System \eqref{wk00} switching along the Markov chain $r_0,r_1,r_2,...,r_n$ gives
\begin{equation}\label{wk02}
\aligned
dX(t)=&\left[-\delta_{r_0} X(t)+G(X(t-\tau_{r_0}), r_0)\right]dt+\sigma_{r_0} X(t)d B(t), \;\;t\in [0,t_0],\\
dX(t)=&\left[-\delta_{r_1} X(t)+G(X(t-\tau_{r_1}), r_1)\right]dt+\sigma_{r_1} X(t)d B(t), \;\;t\in [t_0,t_1],\\[-1mm]
...\\
dX(t)=&\left[-\delta_{r_{n}} X(t)+G(X(t-\tau_{r_n}), r_n)\right]dt+\sigma_{r_n} X(t)d B(t), \;\;t\in [t_{n-1},\Lambda_e),
\endaligned
\end{equation}
where $G(X(t-\tau_{r_i}), r_i)=p_{r_i}X(t-\tau_{r_i})e^{-a_{r_i}X(t-\tau_{r_i})}$. Obviously, each equation is a linear stochastic differential equation, which can be solved step by step, see Mao \cite{mao} (p. 98-99) for more details.
\par In fact, for $t\in[0,\tau_{r_0}]\subset[0,t_0]$, the solution of the first equation of \eqref{wk02} is
\vspace{2mm}
\par
\(\;\;\;\;
X(t)=\Phi_{r_00}(t)\biggl[\phi(0)+\int_0^t\Phi^{-1}_{r_00}(s)G(X(s-\tau_{r_0}), r_0)ds\biggr]>0\;\; a.s.,
\)
\vspace{2mm}
\\
where  $\Phi_{r_0i}(t)=e^{-d_{r_0}(t-i\tau_{r_0})+\sigma_{r_0}[B(t)-B(i\tau_{r_0})]}$, and for $t\in[\tau_{r_0}, 2\tau_{r_0}]\subset[\tau_{r_0},t_0]$ we have
\vspace{1mm}
\par
\(\;\;\;\;
X(t)=\Phi_{r_01}(t)\biggl[X(\tau_{r_0})+\int_{\tau_{r_0}}^t\Phi^{-1}_{r_01}(s)G(X(s-\tau_{r_0}), r_0)ds\biggr]>0\;\; a.s.,
\)
\vspace{1mm}
\par
\(\quad\qquad\;\,
...\)
\vspace{2mm}
\\
and for $t\in [j_0\tau_{r_0},t_0]$, where $j_0=\lfloor\frac{t_0}{\tau_{r_0}}\rfloor$, we get
\vspace{2mm}
\par
\(\;\;\;\;
X(t)=\Phi_{r_0j_0}(t)\biggl[X(j_0\tau_{r_0})+\int_{j_0\tau_{r_0}}^t\Phi^{-1}_{r_0j_0}(s)G(X(s-\tau_{r_0}), r_0)ds\biggr]>0\;\; a.s., \)
\vspace{2mm}
\\
\par For $t\in[t_0, t_0+\tau_{r_1}]\subset[t_0, t_1]$, the solution of the second equation of \eqref{wk02} is
\vspace{2mm}
\par
\(\;\;\;\;
X(t)=\Phi_{r_10}(t)\biggl[X(t_0)+\int_{t_0}^t\Phi^{-1}_{r_10}(s)G(X(s-\tau_{r_1}), r_1)ds\biggr]>0\;\; a.s., 
\)
\vspace{2mm}
\\
where $\Phi_{r_1i}(t)=e^{-d_{r_1}[t-(t_0+i\tau_{r_1})]+\sigma_{r_1}[B(t)-B(t_0+i\tau_{r_1})]}$,  and, for $t\in[t_0+\tau_{r_1}, t_0+2\tau_{r_1}]\subset[t_0+\tau_{r_1}, t_1]$,
\vspace{2mm}
\par
\(\;\;\;\;
X(t)=\Phi_{r_11}(t)\biggl[X(t_0+\tau_{r_1})+\int_{t_0+\tau_{r_1}}^t\Phi^{-1}_{r_11}(s)G(X(s-\tau_{r_1}), r_1)ds\biggr]>0\;\; a.s.,
\)
\vspace{2mm}
\par
\(\quad\qquad\;\,
...\)
\vspace{2mm}
\\
and for $t\in[t_0+j_1\tau_{r_1}, t_1]$ with $j_1=\lfloor\frac{t_1-t_0}{\tau_{r_1}}\rfloor$,
\[
X(t)=\Phi_{r_1j_1}(t)\biggl[X(t_0+j_1\tau_{r_1})+\int_{t_0+j_1\tau_{r_1}}^t\Phi^{-1}_{r_1j_1}(s)G(X(s-\tau_{r_1}), r_1)ds\biggr]>0\;\; a.s. 
\]
Repeating this procedure, we can obtain the explicit solution $X(t)$ of  (\ref{wk02}), which is positive on $[0,\Lambda_e)$ a.s.
 \par {\bf Step 2.} Now, we prove that this solution $X(t)$ to equation \eqref{wk00} with \eqref{wk01} is global existence, that is, the solution will not explode in finite time. We only need to show  $\Lambda_e=\infty$ a.s.
 Let $k_0>0$ be sufficiently large that $\max_{t\in[-\tau,0]}\phi(t)<k_0$, and for each integer $k\geqslant k_0$ define the stopping time
$\Lambda_k=\inf\{t\in[0,\Lambda_e)|\; X(t)\geqslant k\}.$ It is clear that $\Lambda_k$ is increasing as $k\ra\infty.$
Let $\Lambda_\infty=\lim _{k\ra \infty}\Lambda_k$, whence $\Lambda_\infty\leqslant\Lambda_e$ a.s.
 In order to show $\Lambda_e=\infty$,
we only need to prove  that $\Lambda_\infty=\infty$ a.s.   If it is false, then there is a pair of constants $T>0$ and $\epsilon\in(0,1)$ such that
$\mathbb{P}\{\Lambda_\infty\leqslant T\}>\epsilon,$
which yields that  there  exists an integer $k_1\geqslant k_0$ such that
\(\label{n1}
\mathbb{P}\{\Lambda_k\leqslant T\}\geqslant\epsilon\;
\; \mbox{for all}\;\; k\geqslant k_1.
\)

Let
$V(x,i)=x^\theta$ with $\theta\in[1,1+2\rho)$, then
by It\^{o}'s formula, we have
\begin{equation}\label{wangk2}
\aligned dV(X(t),r_t)
=&\left[-\beta_{r_t}X^{\theta}(t)+\theta
X^{\theta-1}(t)G(X(t-\tau_{r_t}),r_t)\right]dt+\theta\sigma_{r_t}
X^{\theta}(t)dB(t)\\
\triangleq& \mathcal{L}V(X(t), r_t)dt+\theta\sigma_{r_t} X^{\theta}dB(t).
\endaligned
\end{equation}
It follows from $X(t)>0$ a.s. on $[0,\Lambda_e)$, $\beta_{r_t}>0$ and  Lemma \ref{lem1} that
$\mathcal{L}V(X(t), r_t)
\leqslant-\beta_{r_t}X^\theta(t)+\theta\gamma_{r_t}X^{\theta-1}(t)\leqslant M_{r_t}.
$
Thus, we have
$
dV(X(t),r_t)\leqslant M_{r_t}dt+\theta\sigma_{r_t} X^{\theta}(t)dB(t).
$
For $t\in[0,\tau_{r_0}]\subset[0,t_0]$, integrating both sides of  above inequality on $[0,\Lambda_n\wedge t]$ and taking expectation lead to
\[
\int_0^{\Lambda_n\wedge t}dV(x(t),r_0)\leqslant\int_0^{\Lambda_n\wedge
t}M_{r_0}dt+\int_0^{\Lambda_n\wedge t}\theta\sigma_{r_0} X^{\theta}(s)dB(s),
\]
then
$
\mathbb{E}[V(\Lambda_n\wedge t,r_0)]\leqslant V(\phi(0))+M_{r_0}\mathbb{E}[(\Lambda_n\wedge  t)]\leqslant
V(\phi(0))+M_{r_0}\tau_{r_0},
$
which yields $\Lambda_\infty\geqslant\tau_{r_0}$ a.s. Similarly, we can obtain that $\Lambda_\infty\geqslant t_0$. Repeating this procedure, we can show $\Lambda_\infty\geqslant t_1$,..., and $\Lambda_\infty\geqslant t_{n-1}+j_n\tau_{r_n}$  a.s. for any  integer $j_n$, which implies $\Lambda_\infty=\infty$ a.s.
 \end{proof}
\begin{theo} \label{Th1} 
 The solution $X(t)$ of equation \eqref{wk00} with initial value \eqref{wk01} has the following properties:
\[\aligned
&\limsup_{t\rightarrow
\infty}\mathbb{E}[X^\theta(t)]\leqslant \hat M, \;\; \limsup_{t\rightarrow
\infty}\frac1t\int_0^t\mathbb{E}[X^\theta(s)]ds\leqslant {\hat W}/\alpha, \;\; \mbox{where}\;\;\theta\in[1,1+2\hat\rho);\;\;\mbox{and}\\
&-\check d \leqslant \liminf_{t\rightarrow\infty}\frac 1t {\log X(t)} \leqslant\limsup_{t\rightarrow\infty}\frac 1t{\log X(t)} \leqslant {\check C}/2\;\;\;
a.s.
\endaligned
\]
\end{theo}
\begin{proof}
It follows from It$\hat{o}$ formula that
$
\mathbb{E}[e^tV(X(t))]\leqslant V(\phi(0))+\int_0^t\hat Me^sds=V(\phi(0))+\hat M(e^t-1),
$
which yields
\(\limsup_{t\ra\infty} \mathbb{E}[X^\theta(t)]\leqslant  \hat M\;\,a.s.\) From equation \eqref{wangk2}  we have
$
dV(t)+\alpha X^\theta(t)dt\leqslant W_{r_t}dt+\theta\sigma_i X^{\theta}(s)dB(s).
$
Integrating this inequality from $0$ to $t$ and taking expectation on both sides of it gives
\[
\alpha \int_0^t \mathbb{E} [X^\theta(s)]ds\leqslant  \int_0^tW_{r_s}ds+V(\phi(0)),
\]
which implies
$\limsup_{t\rightarrow \infty}\frac1t\int_0^t\mathbb{E} [X^\theta(s)]ds\leqslant \hat W/\alpha\;\,a.s.$  Applying It$\hat{o}$ formula we have
\begin{equation}\label{r0}
\log(1+X^{2}(t))=\log(1+\phi^{2}(0))+\int_0^tF(X(s),r_s)ds-2\int_0^t\frac{\sigma^2_{r_s}X^4(s)}{(1+X^{2}(s))^2}ds+M(t),
\end{equation}
where $F(X(t),r_t)=\frac{1}{1+X^{2}(t)}[(\sigma^2_{r_t}-2\delta_{r_t})X^{2}(t)+2p_{r_t}X(t)G(t-\tau_{r_t},r_t)]$ and $M(t)=2\int_0^t\frac{\sigma_{r_s}X^2(s)}{1+X^{2}(s)}dB(s)$. It follows from   Lemmas  \ref{lem1}-\ref{lem2}  that
\begin{equation}\label{r1}
F(X(t),r_t)\leqslant[(\sigma^2_{r_t}-2\delta_{r_t})X^{2}(t)+2\gamma_{r_t}X(t)]/({1+X^{2}(t)})
\leqslant C_{r_t}.
\end{equation}
 Meanwhile, the exponential martingale inequality yields, for any $k\in \mathbb{N}$,
$\mathbb{P}\{\sup_{0\leqslant t\leqslant k}[M(t)-2\langle M(t), M(t)\rangle]\geqslant \log \sqrt{k}\}\leqslant 1/k^{2}.$
It follows from $\sum_k k^{-2}<\infty$ and the Borel-Cantelli lemma that there exists an $\Omega_0\subset\Omega$ with $\mathbb{P}(\Omega_0)=1$ such that for every $\omega\in \Omega_0$ there is an integer $k_0=k_0(\omega)$
 such that, for all $k\geqslant k_0$  and $0\leqslant t\leqslant k$, we get
$ \sup_{0\leqslant t\leqslant k}\left[M(t)-2\langle M(t), M(t)\rangle\right]\leqslant \log \sqrt{k}.$ 
  Substituting this inequality and (\ref{r1}) into equation (\ref{r0}), for all $\omega\in \Omega_0$  and $0\leqslant t\leqslant k$, we obtain
\[
\log(1+X^{2}(t))\leqslant\log(1+\phi^{2}(0))+\int_0^tC_{r_s}ds+\log \sqrt{k}.
\]
Thus, for all $\omega\in\Omega_0$ and $t\in[k-1,k]$ we have
\[
\frac1t\log(1+X^{2}(t))\leqslant\frac1{k-1}\left[\log(1+\phi^{2}(0))+\log \sqrt{k}\right]+\frac1t\int_0^tC_{r_s}ds.
\]
By letting $t\rightarrow\infty$, we get
\[
\limsup_{t\rightarrow\infty}\frac1t\log(X^{2}(t))\leqslant\limsup_{t\rightarrow\infty}\frac1t\log(1+X^{2}(t))\leqslant
\limsup_{t\rightarrow\infty}\frac1t\int_0^tC_{r_s}ds\leqslant \check{C}\;\; a.s.
\]
This leads to the desired assertion. It follows from It$\hat{o}$ formula and equation (1.3) that
\begin{equation}\label{wkll}
\log X(t)=\log \phi(0)-\int_0^t\left(\delta_{r_s}+\frac12\sigma^2_{r_s}\right)ds+\int_0^tp_{r_s}\frac{X(s-\tau_{r_s})e^{-a_{r_s}X(s-\tau_{r_s})}}{X(s)}ds+\int_0^t\sigma_{r_s}dB(s),
\end{equation}
which together with the large number theorem for martingales yields
\(
\liminf_{t\ra \infty}\frac1t\log X(t)\geqslant
-\check d.
\)
\end{proof}
\begin{theo} \label{the1}
Assume (A) holds, then the solution $X(t)$ of equation (\ref{wk00}) with (\ref{wk01}) has the  properties:
\[
\limsup_{t\rightarrow
\infty}\frac1t\int_0^t\mathbb{E}[X^\theta(s)]ds\leqslant\frac1\alpha \sum_{i\in\S}\pi_iW_i,\;\;\mbox{where}\;\;\theta\in[1,1+2\hat\rho);\;\;\mbox{and}\;\;\]\vspace{-4.7mm}
\begin{equation*}\label{r00}
-\sum_{i\in\S}\pi_id_i\leqslant\liminf_{t\rightarrow\infty}\frac 1t {\log X(t)} \leqslant\limsup_{t\rightarrow\infty} \frac 1t {\log X(t)} \leqslant\frac12\sum_{i\in\S}\pi_iC_i\;\; a.s.
\end{equation*}
\end{theo}
\begin{proof} The Markov $r_t$  has an invariant distribution $\pi=(\pi_i, i\in\S)$, thus we have
$
\limsup_{t\rightarrow\infty}\frac1{t}\int_0^tW_{r_s}ds= \sum_{i\in\S}\pi_iW_i,$ $\limsup_{t\rightarrow\infty}\frac1{t}\int_0^tC_{r_s}ds=\sum_{i\in\S}\pi_iC_i$ and
$\limsup_{t\ra \infty}\frac1t\int_0^t(\delta_{r_s}+\frac12\sigma^2_{r_s})ds=\sum_{i\in\S}\pi_id_i.$
\end{proof}
\begin{Rem}
 These estimations presented in Theorem \ref{the1} are dependent on the stationary distribution $\pi$ of the Markov chain $(r_t)_{t\geqslant0}.$ Especially, these estimations are also suitable for the case of $\delta\leqslant\frac12\sigma^2$, which fills in the gap of the corresponding results in  Wang et.al \cite{WangW}.
\end{Rem}

\begin{theo} Assume (A) holds,
then the solution $X(t)$ of model \eqref{wk00} with initial value \eqref{wk01} satisfies
\begin{equation}\label{kl1}
0\leqslant\liminf_{t\ra\infty} X(t)\leqslant \sum_{i\in\S}\pi_i\gamma_i/ \sum_{i\in\S}\pi_id_i:= N^*\;\;a.s.$$
\end{equation}
\end{theo}

\begin{Rem} Inequality \eqref{kl1} means that if the daily egg production rate goes to 0, then the limit inferior of the population of the sheep-blowfly will go to 0. In fact, one can see from \eqref{kl1} that $p_i\ra 0$, $\forall\,i\in\S$ lead to $N^*\ra 0$, which yields
$\inf X(t)\ra 0$ as $t\ra\infty$ a.s.
\end{Rem}

\begin{proof}  If $\liminf_{t\ra\infty} X(t)>N^*$ a.s., then there exists a positive constant $\varepsilon$ such that $\liminf_{t\ra\infty}  X(t)=N^*+2\varepsilon$ a.s., and for this  $\varepsilon$ there exists a constant $T>0$ so that
$X(t)\geqslant  N^*+\varepsilon$ for all $t>T$ a.s. It follows from the It$\hat{o}$ formula and equation (\ref{wk00}) with (\ref{wk01}) that\vspace{-2mm}
\[
\log X(t)\leqslant\log \phi(0)+\int_0^t\biggl(-d_{r_s}+\frac{\gamma_{r_s}}{X(s)}\biggr)ds+\int_0^t\sigma_{r_s}dB(s)\;\;a.s.
\]\vspace{-2mm}
Thus \vspace{-2mm}
\[
\limsup_{t\ra\infty}\frac1t\log X(t)\leqslant\limsup_{t\ra\infty}\frac1t\int_0^t\biggl(-d_{r_s}+\frac{\gamma_{r_s}}{N^*+\varepsilon}\biggr)ds
=\sum_{i\in\S}\pi_i\left(-d_{i}+\frac{\gamma_{i}}{N^*+\varepsilon}\right)<0\;\;a.s.
\]
which yields $\limsup_{t\ra\infty} X(t)=0$, a contradiction.
\end{proof}

\begin{theo} Assume (A) holds, then the solution $X(t)$ of equation (\ref{wk00}) with initial value (\ref{wk01}) has the properties: if $\forall\,i\in\S$, $p_i=0$, then \vspace{-1mm}
\[\limsup_{t\ra\infty}\frac1t\log X(t)\leqslant-\sum_{i\in\S}\pi_id_i<0\;\;\mbox{and}\;\;\lim_{t\ra\infty}  X(t)=0\;\; a.s.;\]
\vspace{-2mm}
 if $\tau_i=0$, $\forall\,i\in\S$, then
\[\limsup_{t\ra\infty}\frac1t\log X(t)\leqslant-\sum_{i\in\S}\pi_i(d_i-p_i):=-\lambda,\] further if  $\lambda>0$,  then $\lim_{t\ra\infty}  X(t)=0$; and if $\lambda=0$, then $\limsup_{t\ra\infty}  X(t)\leqslant1$ a.s.
\end{theo}

\begin{theo} Assume  $\forall\,i\in S$, $\tau_i=\tau$ and $a_i=a$ are independent of the Markov chain $(r_t)_{t\geqslant0}$, and $\hat{\mu}>0$,  then the solution $X(t)$ of equation (\ref{wk00}) with initial value (\ref{wk01}) has the properties: \[\limsup_{t\ra\infty}\frac1t\log \mathbb{E}[X^2(t)]\leqslant-\kappa<0,\; \limsup_{t\ra\infty}\frac1t\log X(t)\leqslant-\kappa/2<0\;\;\mbox{and}\;\;  \lim_{t\ra\infty}  X(t)=0\;\; a.s., \] where $\kappa\in(0, \hat{\mu})$ is the unique root of the equation $\kappa \vartheta \tau  e^{\kappa\tau}+\kappa= \hat{\mu},$ and $\vartheta>\check{p}$ is a constant.
\end{theo}

\begin{proof}
Let $e^{\kappa t}V(x,t)=e^{\kappa t}[x^2+\vartheta\int^t_{t-\tau}g(x(s))ds]$, where $g(x)=x^2e^{-2ax}$,  by applying It$\hat{o}$ formula  we have
\[
\aligned
d[e^{\varepsilon t} V(X(t),t)] 
\leqslant&e^{\kappa t}\biggl[\kappa\vartheta\int^t_{t-\tau}g(X(s))ds+(p_{r_t}+\sigma^2_{r_t}-2\delta_{r_t}+\kappa)X^2(t)\biggr]dt+2e^{\kappa t}\sigma^2_{r_t}X^2(t)dB(t)
\endaligned
\]
Integrating two sides of this inequality from $0$ to $T>0$ and taking expectation we obtain
\begin{equation}\label{ad}
e^{\kappa T} \mathbb{E} V(X(T),T)\leqslant V(\phi(0),0)+\mathbb{E}\int_0^Te^{\kappa t}\biggl[\kappa\vartheta\int^t_{t-\tau}g(X(s))ds+(p_{r_t}+\sigma^2_{r_t}-2\delta_{r_t}+\kappa)X^2(t)\biggr]dt.
\end{equation}
On the other hand, we have
\[\aligned
\int_0^Te^{\kappa t}\int^t_{t-\tau}g(X(s))dsdt&=\int_{-\tau}^T\biggl[\int_{s\vee0}^{(s+\tau)\wedge T}e^{\kappa t}dt\biggr]g(X(s))ds\leqslant\tau e^{\kappa\tau}\int_{-\tau}^T e^{\kappa s}g(X(s))ds\\
&\leqslant\tau e^{\kappa\tau}\int_{0}^T e^{\kappa s}X^2(s)ds+\tau^2 e^{\kappa\tau-2}/a^2,
\endaligned
\] which together with \eqref{ad} gives
\[
e^{\kappa T} \mathbb{E} V(X(T),T)\leqslant V(\phi(0),0)+\kappa\vartheta\tau^2 e^{\kappa\tau-2}/a^2+\mathbb{E}\int_0^Te^{\kappa t}(\kappa \vartheta \tau  e^{\kappa\tau}+\kappa-\hat{\mu})X^2(t)dt\leqslant V(\phi(0),0)+\kappa\vartheta\tau^2 e^{\kappa\tau-2}/a^2.
\]
Hence, \[e^{\kappa T} \mathbb{E}[X^2(T)]\leqslant V(\phi(0),0)+\kappa\vartheta\tau^2 e^{\kappa\tau-2}/a^2,\]
 which implies $\limsup_{t\ra\infty}\frac1t\log(\mathbb{E}[X^2(t)])\leqslant-\kappa.$
 \par The rest of the proof uses the same arguments as those in the proof of Mao \cite{mao} (p.373-374).
For completeness, we include the details here.
For arbitrary $\epsilon\in(0, \kappa/2)$, there must exist a constant $K>0$ so that
\[\mathbb{E}[X^2(t)]\leqslant Ke^{-(\kappa-\epsilon)t}, \;\forall\, t\geqslant-\tau.\] For $n=1,2,...$, by the Doob martingale inequality and Holder inequality, we get
\[\mathbb{E}[\sup_{n\leqslant t\leqslant n+1}X^2(t)]\leqslant 3\mathbb{E}[X^2(n)]+C\int_n^{n+1} \left(\mathbb{E}[X^2(t)]+\mathbb{E}[X^2(t-\tau)]\right)dt,\]
 where $C$ is  positive constant independent of $n$, and allowed to be different in deferent lines. Therefore, we obtain
 $\mathbb{E}[\sup_{n\leqslant t\leqslant n+1}X^2(t)]\leqslant Ce^{-(\kappa-\epsilon)n},$ which together with Chebyshev's theorem yields
 \[\mathbb{P}\{\omega: \sup_{n\leqslant t\leqslant n+1}X^2(t)> e^{-(\kappa-2\epsilon)n}\}\leqslant C e^{-\epsilon n}.\]
  It follows from the Borel-Cantelli lemma that for almost all $\omega\in\Omega$ there is a random integer $n_0(\omega)$ such that for all $n\geqslant n_0,$
\[\sup_{n\leqslant t\leqslant n+1}X^2(t)\leqslant e^{-(\kappa-2\epsilon)n},\]
 which yields $\limsup_{t\ra\infty}\frac1t\log X(t)\leqslant-\kappa/2+\epsilon$ a.s. The arbitarity of $\epsilon$  implies the desired assertion.
\end{proof}
\section{Examples}
In this section  an example is given  to check the results obtained in section 2 by numerical simulation.
\par Let $\S=\{1,2,3\}$, $Q=[-10, 4, 6; 2, -3, 1; 3, 5, -8]$, $\delta=[2, 1, 4]$, $p=[4, 2, 8]$,
$\tau=[1, 1, 1]$,  $a=[0.4, 0.2, 0.3]$, $\sigma=[1.5, 2, 3]$,
and $\phi(t)=1$, $t\in[-1, 0]$, then the Markov chain $r_t$ is irreducible and has a unique stationary distribution
$\pi=(0.1845, 0.6019, 0.2136)$ and
one can see that $\delta_1>\sigma_1^2/2,$ $\delta_2=\sigma_2^2/2,$ and $\delta_3<\sigma_3^2/2$. It follows from Theorem 2.1 that the solution of equation \eqref{wk00} is global existence and positive almost surely on $\R^+_0$, Figure 2(left) shows this. Meanwhile, we get
$\limsup_{t\rightarrow \infty} \mathbb{E}[X(t)]\leqslant  1.1221$
and $\limsup_{t\rightarrow \infty} \mathbb{E}[X^{7/5}(t)]\leqslant 0.3713$. Furthermore, from Theorem 2.3 we have $C\thickapprox[2.1022,3.6788,10.3229]$, $d=[3.125, 3, 8.5]$, $N^*= 1.1883$ and the estimation of the sample Lyapunov exponent:
$-4.1978\leqslant\liminf_{t\rightarrow\infty}\frac 1t{\log X(t)}\leqslant\limsup_{t\rightarrow\infty}\frac1t{\log X(t)}\leqslant2.4035$.
One can see from Figure 2(right) that the  value of $\frac1t\log X(t)$ oscillates up and down at zero as $t\rightarrow \infty$, which yields that the sample path of $X(t)$ oscillates up and down at 1 as $t\rightarrow \infty$. Thus, $X(t)\nrightarrow 0$ as $t\rightarrow \infty$, that is, the population  will neither go to extinction finally nor explode in finite time.
\par In order to verify the effectiveness of Theorem 2.6, we reset $p=[0.2, 0.2, 0.4]$, $a=[0.4, 0.4, 0.4]$, $\sigma=[1.5, 1, 2.5]$ then  $\hat\mu>0$  holds. One can see from Figure 1(right) that $X(t)\ra0$ and $\frac1t\log X(t)<0$ as $t\ra\infty.$
\vspace{-1mm}
\begin{center}
\scalebox{0.52}
{\includegraphics{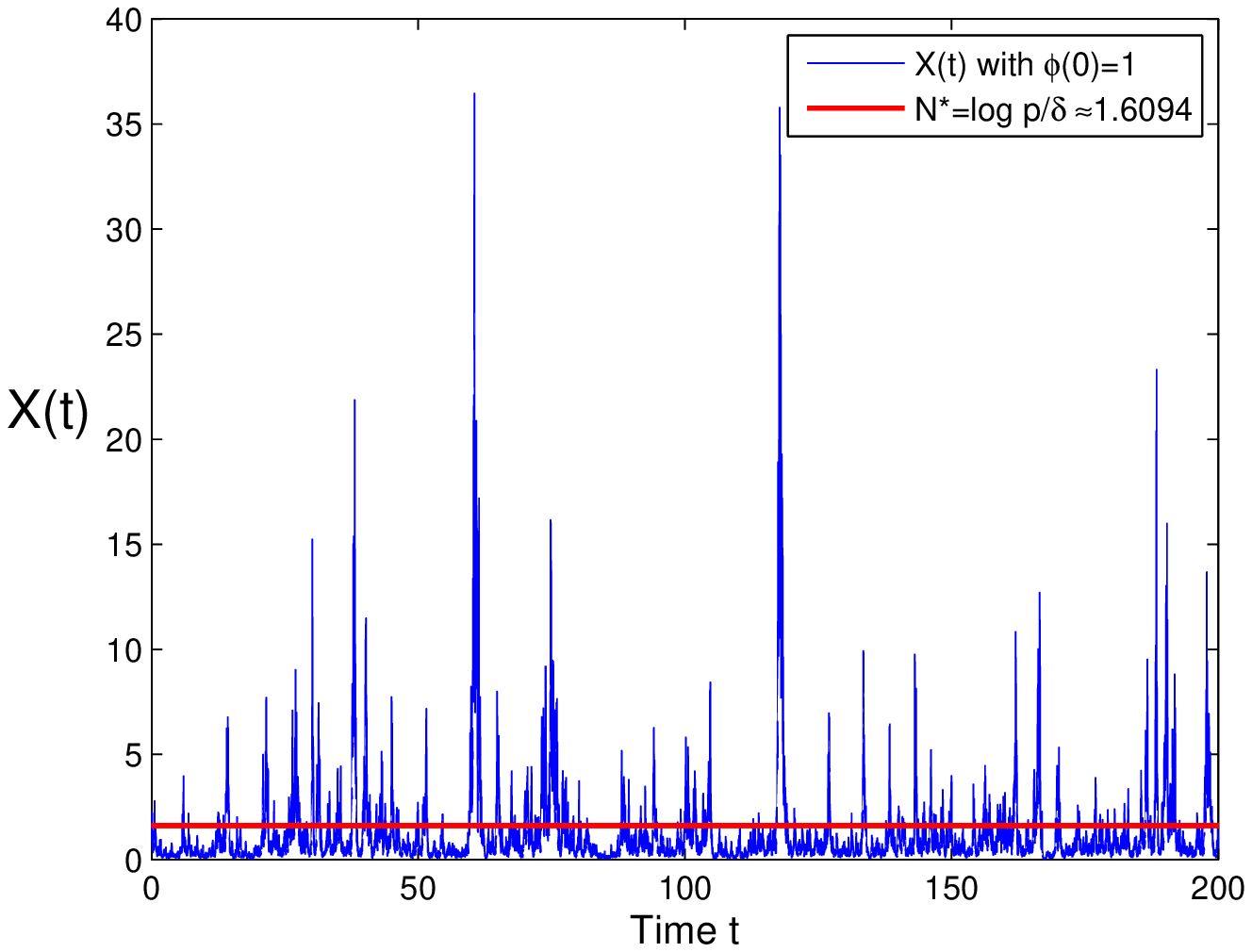}}\scalebox{0.52}{\includegraphics{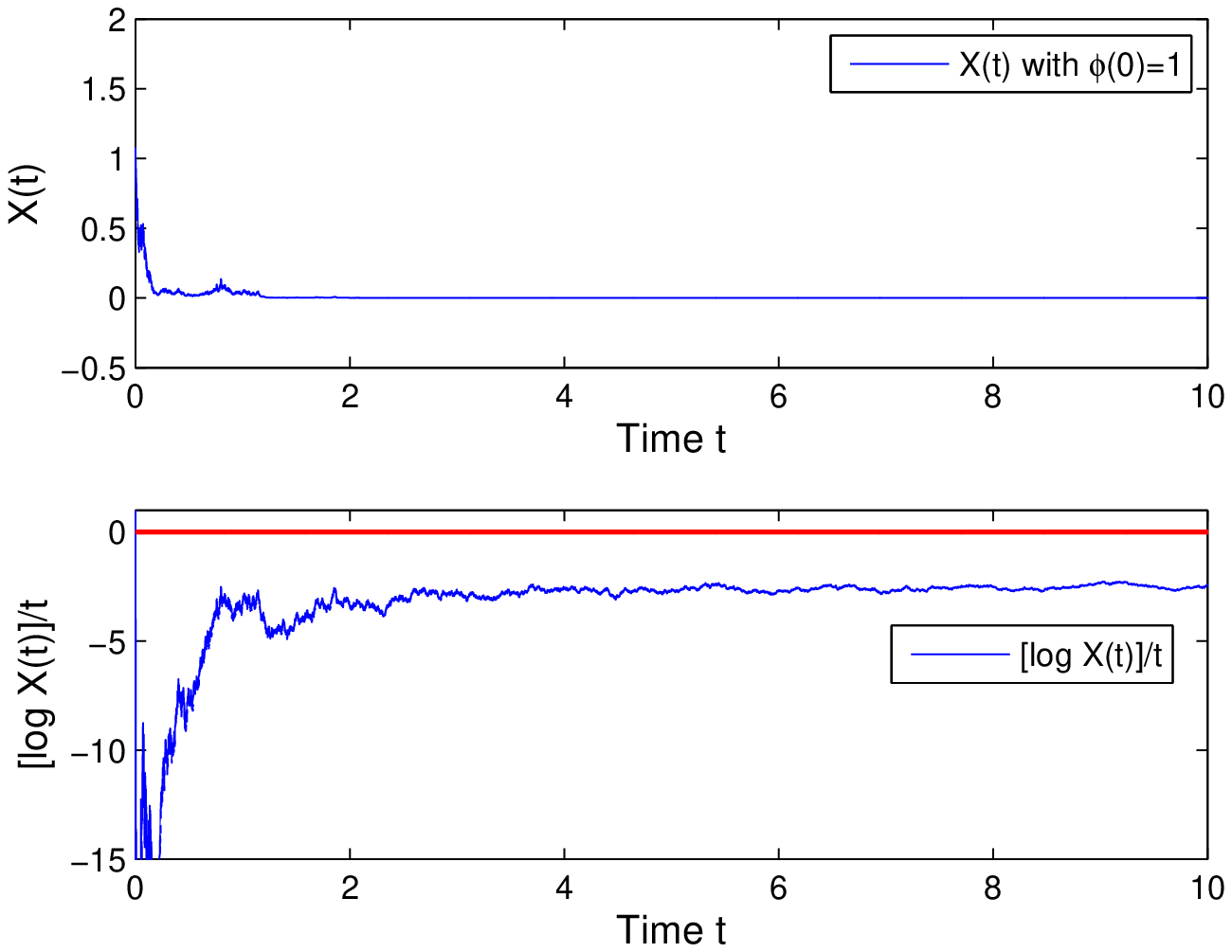}}\\
\footnotesize {\textbf{Figure 1.}\;\;
Left: The solution $X(t)$ of equation \eqref{wk0} with initial value $\phi(t)=1$ for $t\in[-1,0]$ and the step $\Delta=1e-4$.
  Right(up): Simulation of the processes $(X(t), r_t)$ with $r_0=3$ and  $\hat\mu>0$. Right(down): $[\log X(t)]/t$ of the sample path $X(t)$.}
\end{center}
\begin{center}
\scalebox{0.52}
{\includegraphics{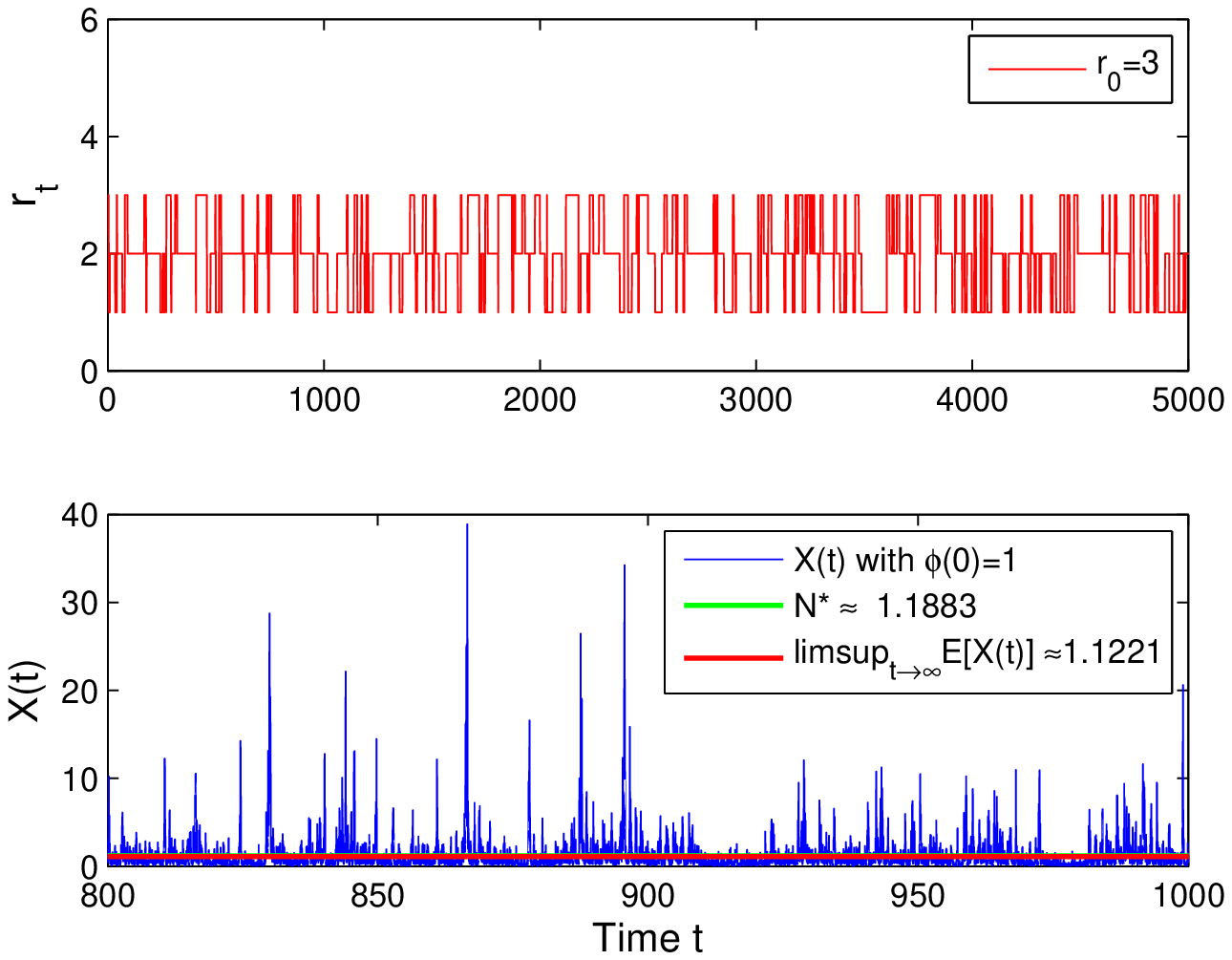}}\scalebox{0.52}{\includegraphics{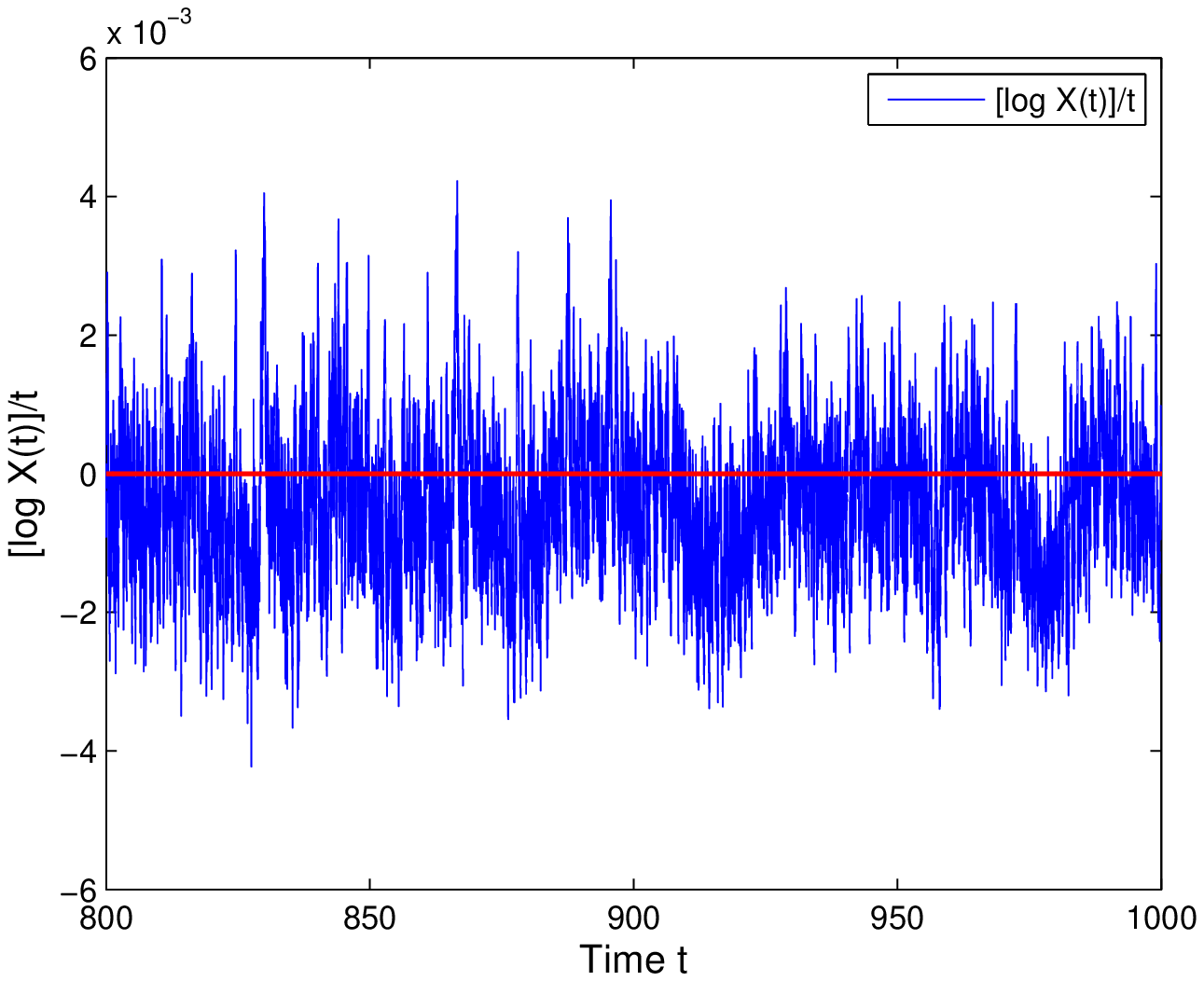}}\\
\footnotesize {\textbf{Figure 2.}\;\; Here the step $\Delta=1e-4$. Left: Simulation of the processes $(X(t), r_t)$ with $r_0=3,$} and $\phi(t)=1$ for $t\in[-1,0]$ (In order to see the line clearly, we only present the first $5\times 10^3$ states of $r_t$, and just show the values of $X(t)$ on interval $[800, 1000]$). Right: $[\log X(t)]/t$ of the sample path $X(t)$ 
(For clarity,  we just print the values of $[\log X(t)]/t$ on interval $[800, 1000]$).
\end{center}

\section{Conclusion}
In this paper, we consider a stochastic Nicholson's blowflies delay differential equation with regime switching. To the best of our knowledge it is the first time that regime switching is considered in Nicholson's blowflies model. By choosing several suitable $V$
functions we prove that the solution $X(t)$ of the stochastic Nicholson's blowflies model is globally existent, which is also positive on $\R^+_0$ almost surely, and present some estimation of the solution $X(t)$ to model (\ref{wk00})-(\ref{wk01}).   The results obtained in this paper extend some corresponding  results in literature, especially we removed the restricted condition $\delta>\sigma^2/2$, which is needed in \cite{WangW}. It is worth to point out that model \eqref{wk00}  can be extended to a general case:
\[
dX(t)=\left[-\delta_{r_t} X(t)+G(X(t-\tau_{r_t}), r_t)\right]dt+\sigma_{r_t} X(t)d B(t),
\]
where $G: \mathbb{R}\times \S\rightarrow\mathbb{R}^+_0$ is uniformly bounded, and the methods applied in Theorems 2.1$-$2.5 are also suitable for it.
\par  For model \eqref{wk00} with initial value \eqref{wk01}, there are still some interesting problems left:
\par 1) Find  conditions for the almost surely strong (weak) persistence of the population, that is, under what conditions  $\liminf_{t\ra \infty}$ $ X(t)>0$ a.s. ($\limsup_{t\ra\infty} X(t)>0$ a.s.); or more accurately, estimate the maximum value of $K\in \mathbb{R}^+$ such that  $\liminf_{t\ra \infty}X(t)\geqslant K$ a.s.
\par 2) In the case of $p>0$, $\tau>0$, the delay $\tau$ and parameter $a$ vary along the Markov chain $(r_t)_{t\geqslant0}$, that is, they depend on the Markov chain $(r_t)_{t\geqslant0}$. Find conditions for  the upper boundedness of the population, that is, under what conditions  there exists a $M\in \mathbb{R}^+$ such that $\limsup_{t\ra\infty} X(t)\leqslant M$ a.s., and give the estimation of this upper bound $M$. Find conditions for the negativeness of the Lyapunov exponent of $X(t)$, that is  $\limsup_{t\ra\infty}\frac1t\log X(t)<0$ a.s.

\end{document}